\documentclass[11pt,a4paper]{amsart}

\usepackage{geometry}
\pdfoutput=1
\usepackage{amsmath,amsfonts,amssymb,amsthm}
\usepackage{xcolor}
\usepackage{todonotes}
\usepackage{hyperref}
\hypersetup{
	colorlinks   = true,
	pdfauthor    = {Vitezslav Kala and Pavlo Yatsyna},
	pdftitle     = {},
	pdfkeywords  = {},
	pdfcreator   = {Pdflatex},
	pdfpagemode  = UseNone,
	pdfstartview = FitH
}
\usepackage{pdfpages}
\usepackage[normalem]{ulem}
\usepackage{url}
\usepackage{enumerate}
\usepackage{verbatim}
\usepackage{mathrsfs}
\usepackage{csquotes}
\usepackage{mathtools}
\usepackage{lipsum}
\usepackage{enumitem}

\makeatletter
\newcommand*{\house}[1]{%
	\mathord{%
		\mathpalette\@house{#1}%
	}%
}
\newcommand*{\@house}[2]{%
	\dimen@=\fontdimen8 %
	\ifx#1\scriptscriptstyle\scriptscriptfont
	\else\ifx#1\scriptstyle\scriptfont
	\else\textfont\fi\fi
	3 %
	\sbox0{%
		$#1%
		\vrule width\dimen@\relax
		\overline{%
			\kern2\dimen@
			\begingroup 
			#2%
			\endgroup
			\kern2\dimen@
		}%
		\vrule width\dimen@\relax
		\mathsurround=1.5\dimen@ 
		$%
	}%
	\ht0=\dimexpr\ht0-\dimen@\relax
	\dp0=\dimexpr\dp0+2\dimen@\relax
	\vbox{%
		\kern\dimen@ 
		\copy0 %
	}%
}

\theoremstyle{plain}
\newtheorem{theorem}{Theorem}
\newtheorem{lemma}[theorem]{Lemma}

\newtheorem{cor}[theorem]{Corollary}

\newtheorem*{lemma*}{Lemma}
\theoremstyle{definition}

\newtheorem{definition}[theorem]{Definition}

\newcommand{\R}{\mathbb{R}}
\newcommand{\N}{\mathbb{N}}
\newcommand{\Z}{\mathbb{Z}}
\newcommand{\Q}{\mathbb{Q}}
\newcommand\Nr{\mathrm{N}}

\newcommand{\co}{\mathcal{O}}

\usepackage{tikz}

\newcommand*\cO{\mathcal{O}}
\newcommand*\ok{\cO_K}

\DeclareMathOperator\Tr{Tr}

\DeclarePairedDelimiterX\set[2]\lbrace\rbrace{\,#1\mathclose{}:\mathopen{}#2\,}

\title[On Kitaoka's conjecture and lifting problem for universal quadratic forms]{On Kitaoka's conjecture and lifting problem\\ for universal quadratic forms}

\author{V\' \i t\v ezslav Kala}
\address{Charles University, Faculty of Mathematics and Physics, Department of Algebra, Sokolov\-sk\' a 83, 18600 Praha~8, Czech Republic}
\email{{vitezslav.kala@matfyz.cuni.cz}}

\author{Pavlo Yatsyna}
\address{Charles University, Faculty of Mathematics and Physics, Department of Algebra, Sokolov\-sk\' a 83, 18600 Praha~8, Czech Republic} 
\address{Aalto University, Department of Mathematics and Systems Analysis, P.O. Box 11100, FI-00076, Finland}
\email{{pavlo.yatsyna@aalto.fi}}

\thanks{Both authors were supported by the project PRIMUS/20/SCI/002 from Charles University.
	V.K. was supported by Czech Science Foundation (GA\v CR) grant 21-00420M and Charles University Research Centre program UNCE/SCI/022. P.Y. was supported by the Academy of Finland (grants 336005 and 351271, PI C. Hollanti), and by MATINE, Ministry of Defence of Finland (grant 2500M-0147, PI C. Hollanti).}
\date{\today}

\begin{document}

	\begin{abstract}
		For a totally positive definite quadratic form over the ring of integers of a totally real number field $K$, we show that there are only finitely many totally real field extensions of $K$ of a fixed degree over which the form is universal (namely, those that have a short basis in a suitable sense). Along the way we give a general construction of a universal form of rank bounded by $D(\log D)^{d-1}$, where $d$ is the degree of $K$ over $\Q$ and $D$ is its discriminant.
		Furthermore, for any fixed degree we prove (weak) Kitaoka's conjecture that there are only finitely many totally real number fields with a universal ternary quadratic form.  
	\end{abstract}

	\maketitle

	\section{Introduction}

	Let $F$ be a totally real number field with the ring of integers $\co_F$. A classical question asks when an algebraic integer $\alpha\in\co_F$ 
	can be written as the sum of squares of elements of $\co_F$; an obvious necessary condition is that $\alpha$ is totally positive (i.e., all its conjugates are positive). This condition is also sufficient in the basic case $F=\Q$ by the four square theorem, and when $F=\Q(\sqrt{5})$. 
	However, Siegel \cite{Si3} proved that these are the only such totally real number fields, motivating the study of two natural generalizations.
	
	The first one asks how many squares are required, provided that $\alpha$ is indeed representable as the sum of squares.
	This leads to the Pythagoras number, a constant well-studied also in other settings (see, e.g., \cite{Le}). 
	Yet, in the case of the ring of integers of a totally real number field, all that is known in general is that the Pythagoras number is finite \cite{Sc2} and bounded by the degree of the field \cite{KY1}, but can grow arbitrarily large \cite{Sc2} (in the non-totally real case, the Pythagoras number $\leq 5$ \cite{Pe2}; for some small degree cases see \cite{Pe, Ti, KRS}). 
	Furthermore, Siegel \cite{Si2} proved that for each number field $F$ there exists $m\in \N$ such that all totally positive integers divisible by $m$ can be represented as the sum of squares. 
	
	The second generalization is to universal forms over $F$, i.e., quadratic forms that represent all the totally positive elements of $\co_F$. 
	The smallest possible rank of a universal form is three, and this can happen only in number fields of even degree  \cite{EK}. Even for them, Kitaoka formulated his influential conjecture that there are only finitely many number fields $F$ admitting a ternary universal form. 
	The only known examples are in real quadratic fields \cite{CKR} and Kitaoka's conjecture remains unproven, even though it seems that universal forms typically must have large ranks \cite{BK2, Ka, KKP, KT, Ya}.
	
	Indecomposable algebraic integers turned out to be one of the key tools in the recent advances on this topic.
	In this short paper, we first prove in Theorem \ref{thm:ind_bound} that each indecomposable has norm smaller or equal to the discriminant of $F$, significantly extending previous results in the quadratic and cubic cases \cite{DS, KT, TV}, 
	as well as improving the previous general bound \cite{Br}, which is typically much worse than ours, as it depends on the regulator. 
	Our result is also a substantial step towards answering \cite[Problem 53]{Na}. 
	
	Theorem \ref{thm:ind_bound} directly implies a construction of a universal quadratic form -- giving the first completely general result in the train of thought started by B.~M. Kim \cite{Ki2}. The following theorem is more precisely formulated and proved as Theorem \ref{thm:univ}.
	
	\begin{theorem}\label{th:0}
		Let $F$ be a totally real number field of degree $d=[F:\Q]$ and discriminant $\Delta$. Then there is a universal quadratic form over $\co_F$ of rank $\ll \Delta(\log\Delta)^{d-1}$.
	\end{theorem}
	
	In the case when the number field $F$ is monogenic (i.e., has integral power basis), contains no proper subfields, and has units of all signatures, 
	there is also a lower bound \cite[Theorem~1.4]{Ya} on the rank of universal quadratic forms that depends on $\Delta$. In particular, for such real quadratic number fields \cite[Theorem~5.4]{Ya}, and for the simplest cubic fields \cite[Theorem~1.1]{KT}, the rank of the universal quadratic form is $\gg \Delta^{1/2}$ or $\gg \Delta^{1/4}$, depending on whether the quadratic form is classical or not, respectively. 
	
	\
	
	Then we move on to the main goal of this note, namely, to proving several finiteness results concerning the (non-)universality of certain quadratic forms (and, slightly more generally, $\co_F$-lattices).
	
	\begin{theorem}\label{thm:1}
		Let $F$ be a totally real number field, $L$ an $\co_F$-lattice, and $d,m\in\N$. 
		There are at most finitely many totally real number fields $K\supset F$ of degree $d=[K:\Q]$ such that $L\otimes\ok$
		represents all elements of $m\ok^+$.
	\end{theorem}
	
	This theorem includes, as a significant special case, a new result on the sum of squares, for one can take $F=\Q$ and $L=\Z^r$ equipped with the quadratic form $Q=x_1^2+\dots+x_r^2$. 
	It also extends the previous results on the \textit{lifting problem} (i.e., whether a form can be universal over a larger number field, see \cite{KY1} and \emph{potentially universal quadratic forms} in \cite{XZ}), and partly resolves an open question formulated in \cite{KY1}:
	As Corollary \ref{cor:z-forms} we show that a given quadratic form is universal only over finitely many totally real number fields of degree $d$. 
	
	We further apply these results to prove a weak version of Kitaoka's conjecture:
	
	\begin{theorem}\label{thm:2}
		For each $d\in\N$, there are only finitely many totally real number fields $K$ of degree $d=[K:\Q]$ over which there is a ternary universal $\ok$-lattice.
	\end{theorem}
	
	For \textit{classical} lattices, this theorem was previously proved by B. M. Kim in an unpublished manuscript \cite{Ki3}.
	
	We begin the article by introducing notation and basic tools in the next section. Section \ref{sec:indeco} considers indecomposables and contains the proofs of Theorems \ref{thm:ind_bound} and \ref{thm:univ}. Theorem \ref{thm:1} is then proved in Section \ref{sec:2} as Theorem \ref{th:lifting}. There we also cover the cases of representations of $m\co_F^+$ by the sums of squares (Corollary \ref{cor:sums of squares}) and the lifting problem for positive definite quadratic forms over $\Z$ (Corollary \ref{cor:z-forms}).  
	Finally, Theorem \ref{thm:2} is proved in the last section as Theorem \ref{th:kitaoka}. There we use the existence of  
 ``universality criterion sets'' (as introduced in \cite{EKK}), which was a folklore result \cite{Km}, whose full proof appeared in the work of Chan and Oh \cite{CO}.
	
	\section*{Acknowledgments}
	We are grateful to Byeong Moon Kim for sharing his manuscript \cite{Ki3} with us, to Giacomo Cherubini and Dayoon Park for several useful discussions and comments, 
	and to Jakub Kr\' asensk\' y for suggesting the specific formulation of Theorem \ref{thm:univ} and for a number of helpful corrections. We are also thankful to the anonymous referee for several very useful suggestions.

	\section{Preliminaries}\label{sec:prel}
	
	Let $F$ be a totally real number field of degree $d$ over $\Q$ and let $\co_F$ denote its ring of integers. Let $\sigma_1,\sigma_2,\dots,\sigma_d:K\hookrightarrow\R$ be the distinct real embeddings of $F$ and let $\sigma=(\sigma_1,\sigma_2,\dots,\sigma_d):F\hookrightarrow\R^d$ be the corresponding embedding of $F$ into the Minkowski space.
	
	The \textit{norm} and \textit{trace} are $\Nr_{F/\Q}, \Tr_{F/\Q}:F\rightarrow \R$, $\Nr_{F/\Q}(\alpha)=\sigma_1(\alpha)\cdots\sigma_d(\alpha)$ and $\Tr_{F/\Q}(\alpha)=\sigma_1(\alpha)+\dots+\sigma_d(\alpha)$.
	
	As a height function on $F$, we will primarily work with the \textit{house} (also called \textit{the maximum modulus of conjugates}), defined as
	$\house{\alpha}_F=\house{\alpha}= \max_i(|\sigma_i(\alpha)|)$.  
	For $v=(v_1,\dots,v_n)^t\in F^n$, we further let $\house{v}_F=\house{v}=\max_j\house{v_j}_F$. We have $\house{\alpha}\ge 1$ for all $\alpha \in \co_F\setminus\{0\}$. 
	
	An element $\alpha\in F$ is \textit{totally positive} if $\sigma_i(\alpha)>0$ for all $i$. The set of all totally positive elements of $F$ is denoted $F^+$. For $\alpha,\beta\in F$, $\alpha$ is \textit{totally greater than} $\beta$ (denoted $\alpha\succ\beta$) if $\alpha-\beta\in F^+$.
	If $E$ is a subset of $F$, then $E^+=E\cap F^+$. 
	
	An element $\alpha\in \co_F^+$ is \textit{indecomposable} if there does not exist $\beta\in \co_F^+$ such that $\alpha\succ \beta.$ 
	
	For a positive integer $m$, let $\sum^m\co_F^{(2)}=\{\sum^m\alpha_i^2:\alpha_i\in\co_F\}$; also $\sum^\infty\co_F^{(2)}=\bigcup_m\sum^m\co_F^{(2)}$. We say that $n=\mathcal P(\co_F)$ is the \textit{Pythagoras number} of $\co_F$ if it is the smallest positive integer (or $\infty$) such that $\sum^n\co^{(2)}_F=\sum^{\infty}\co_F^{(2)}$.
	
	\medskip
	
	We follow the lattice-related terminologies and notations from \cite{O1}. 
	Specifically, let $V$ be $r$-ary \textit{quadratic space} over $F$ with its symmetric \textit{bilinear form} $B:V\times V\longrightarrow F$ and the corresponding \textit{quadratic form} $Q$, i.e., $B(v,v)=Q(v)$. 
	The \textit{Gram matrix} of vectors $v_1,\dots,v_k\in V$ is the $k\times k$ symmetric matrix $(B(v_i,v_j)).$
	
	A \textit{(quadratic)} $\co_F$-\textit{lattice} $L$ on $V$ is a finitely generated $\co_F$-submodule of $V$ such that  
	$FL=V$, which we view equipped with the restrictions of $B$ and $Q$; we often denote this by saying that $(L,Q)$ is an $\co_F$-lattice. 
	The \textit{rank} of $L$ is $r$.
	
	The \textit{scale} of $L$ is $\mathfrak{s}L=\{B(v,w):v,w \in L\},$ while the $\co_F$-module generated by $Q(L)=\{Q(v):v \in L\}$, denoted by $\mathfrak{n}L$, is called the \textit{norm} of $L.$ Throughout the work we assume that $\mathfrak{n}L\subset \co_F$ (i.e., $Q(v)\in\co_F$ for all $v\in L$); then $\mathfrak{s}L\subset \frac 12\co_F$ and $2B(v,w)\in\co_F$ for all $v,w\in L$.
	We say that the scale is \textit{integral} if $\mathfrak{s}L\subset \co_F$, in this case we say that the lattice  is \textit{classical}. 
	By \cite[81:3]{O1} there is a basis $v_1,\dots,v_r$ of $V$ and fractional ideals $\mathfrak{a}_i\subset F$ such that $L=\mathfrak{a}_1v_1+\cdots+\mathfrak{a}_rv_r$. 
	The \textit{volume} of $L$ is $\mathfrak{v}L=\mathfrak{a}^2_1\cdots\mathfrak{a}^2_r\det(B(v_i,v_j))$. 
	
	The quadratic form $Q$ (or, the lattice $L$) is \textit{totally positive definite} if $Q(v)\succ 0$ for all non-zero $v \in V$. 
	In that case we have the \textit{Cauchy--Schwarz inequality} $Q(v)Q(w)\succeq B(v,w)^2$ for all $v,w\in V$ (that easily follows from the usual Cauchy--Schwarz inequality for quadratic forms over $\R$).
	We say that $L$ (or $Q$) \textit{represents} an algebraic integer $\alpha$  if there exists $v \in L$ such that $Q(v)=\alpha$. We say that a totally positive lattice is \textit{universal} (over $\co_F$ or over $F$) if it represents all the elements of $\co_F^+$.

	Given a field extension $K\supset F$ and an $\co_F$-lattice $(L,Q)$, we have the \textit{tensor product} $\co_K$-lattice $L\otimes \co_K$
	defined as $L\otimes\co_K=\co_K\mathfrak{a}_1v_1+\cdots+\co_K\mathfrak{a}_rv_r$ equipped with the natural extensions of $B$ and $Q$, e.g., 
	if 	$x=\sum x_iv_i\in L\otimes\co_K,$ then $Q(x)=x^tMx$ where we identify $x$ with the column vector $(x_1,\dots,x_r)^t$ and $M=(B(v_i,v_j))$ is the Gram matrix. If we moreover identify $V=F^r$, then $L\otimes \co_K\subset K^r$.

	For a quadratic form $Q:\co_F^r\rightarrow \co_F$, we apply all the preceding terminology when it applies to the free $\co_F$-lattice $(\co_F^r,Q)$.
	
	\medskip
	
	\noindent
	\textbf{Convention.} 	
	Whenever we talk about an \textit{$\co_F$-lattice} throughout the paper, we always mean a totally positive definite quadratic $\co_F$-lattice satisfying $\mathfrak nL\subset \co_F$.
	
	\medskip
	
	We shall also use the common asymptotic notations: 
	If $f(x), g(x)$ are two positive real functions, then $f\ll g$ (and $g\gg f$) if there is a constant $C>0$ such that $f(x)<Cg(x)$ for all $x$ (that lie in the domains of $f,g$),  
	and $f\asymp g$ if $f\ll g$ and $g\ll f$.
	
	\medskip
	
	Let us now begin with the following lemma:
	
	\begin{lemma}\label{lem:RR} Let $K\supset F$ be totally real number fields and
		$(L,Q)$ an $\co_F$-lattice. There is $C=C_{L,F}\in\R^+$  such that for all $v\in L\otimes \co_K$ we have $\house{v}_K^2\leq C\house{Q(v)}_K$. 
	\end{lemma}

Note that the constant $C=C_{L,F}$ above depends on $L, Q, F$, and the choice of pseudo-basis $L=\mathfrak{a}_1u_1+\cdots+\mathfrak{a}_ru_r$, but \textbf{not} on $K$.

	\begin{proof}
		This follows from Rayleigh--Ritz Theorem (Rayleigh quotient) \cite[Theorem~4.2.2]{HJ}, that says: \textit{If $M$ is a symmetric matrix with entries in $\R$ and $\lambda_1(M)$ is its smallest eigenvalue, then}
		\begin{equation}\label{eq:lemma}
		\lambda_1(M)\sum v_k^2 \le v^tMv
		\end{equation}
		\textit{for any vector $v=(v_1,\ldots,v_d)^t\in \R^d$.} 
		Let $L=\mathfrak{a}_1u_1+\cdots+\mathfrak{a}_ru_r$ and let $M=(B(u_i,u_j))$ be the corresponding Gram matrix.
		Define $M_i=\sigma_i(M)$ for each embeddings of $F$ in $\R$. In particular, $M_i$ is a symmetric matrix with coefficients in $F$. 
		For each embedding $\sigma_i:F\hookrightarrow\R$, let $\sigma_{ij}:K\hookrightarrow \R$ be all its extensions.
		
		It now suffices to consider the inequality \eqref{eq:lemma} 
		with $C=\frac{1}{\lambda}$, where $\lambda=\min_i \lambda_1(M_i)$ (note that $\lambda>0$, for $L$ is totally positive definite, which implies that each matrix $M_i$ is positive definite). That is,
		\begin{align*}
		\house{Q(v)}_K= &\max_{i, j} \sigma_{ij}(v^tMv)	=  \max_{i,j}\left(  (\sigma_{ij}v)^tM_i(\sigma_{ij}v)\right) 
		\ge \max_{i,j} \left( \lambda_1(M_i)\left(\sum (\sigma_{ij}v_k)^2\right)\right) \\
		\ge & \frac 1C\max_{i,j} \sigma_{ij}\left(\sum v_k^2\right)
		\ge \frac 1C\house{v}_K^2.\qedhere
		\end{align*}
	\end{proof}

\

	\section{Indecomposables}\label{sec:indeco}
	
	Let us start by proving a general bound on the norm of indecomposables.

	\begin{theorem}\label{thm:ind_bound}
		Let $K$ be a totally real number field with discriminant $\Delta.$ 
		For every element $\alpha \in \ok^+$ with $\Nr_{K/\Q}(\alpha)>\Delta$ there is $\beta\in\ok$ such that $\alpha\succ\beta^2$. In particular, no such element $\alpha$ is indecomposable. 
	\end{theorem}
	\begin{proof}
		Let $K$ be of degree $d$ over $\Q$. 
		In the Minkowski space associated to $K$, consider the box defined by $|x_i|\leq \sqrt{\sigma_i(\alpha)}-\varepsilon$, 
		where $\sigma_i$ are the embeddings of $K$ into $\R$ and $\varepsilon>0$ is small enough that \[\prod^d_{i=1}(\sqrt{\sigma_i(\alpha)}-\varepsilon)>\sqrt{\Delta}.\]
		This is possible because $\prod^d_{i=1}\sqrt{\sigma_i(\alpha)}=\sqrt{\Nr_{K/\Q}(\alpha)}>\sqrt{\Delta}$.
		Thus the volume of the box is bigger than $2^d\sqrt{\Delta}$. By Minkowski theorem (e.g., see Theorem III in Chapter III in \cite{Ca}) there exists a non-zero lattice point in this box, which corresponds to an algebraic integer $\beta\in\ok$. We have that $\sqrt{\sigma_i(\alpha)}>\sigma_i(\beta),$ and thus $\sigma_i(\alpha)>\sigma_i(\beta^2)$ for all $i$. Therefore $\alpha\succ \beta^2,$ as was required to show. 
	\end{proof}
	
	As an important corollary, we get the following general construction of a universal quadratic form.
	
	\begin{theorem}\label{thm:univ}
		Let $K$ be a totally real number field of degree $d=[K:\Q]$ whose discriminant is $\Delta$ and Pythagoras number of the ring of integers $\ok$ is $\mathcal P(\ok)=P$. Fix a set of representatives $\mathcal S$ for classes of elements $\alpha\in\ok^+$, $\Nr_{K/\Q}(\alpha)\leq \Delta,$ up to multiplication by squares of units in $\ok$. Then the diagonal quadratic form
		$$Q=\sum_{\alpha\in\mathcal S}\alpha x_{\alpha}^2+y_1^2+\dots+y_{P}^2$$
		is universal and has rank $\#\mathcal S+P\ll \Delta(\log\Delta)^{d-1}$ (where the implied constant depends only on $d$).
	\end{theorem}
	
	\begin{proof}
		Let $\gamma\in\ok^+$. If $\Nr_{K/\Q}(\gamma)>\Delta$, then we can repeatedly use Theorem \ref{thm:ind_bound} to find elements $\gamma_0\in\ok^+,\beta_1,\dots,\beta_t\in\ok$ such that $\gamma=\gamma_0+\beta_1^2+\dots+\beta_t^2$ and $\Nr_{K/\Q}(\gamma_0)\leq \Delta$.
		Then $\gamma_0=\alpha x^2$ for some $\alpha\in\mathcal S$ and $x\in\ok^\times$, and the sum of squares $\beta_1^2+\dots+\beta_t^2$ is represented as $y_1^2+\dots+y_{P}^2$ by the definition of the Pythagoras number $P$.		
		
		Now it remains to estimate $\#\mathcal S+P$: 
		The size of $\mathcal S$ is at most $2^d$-times the number of principal ideals of norm $\leq \Delta$ (for in $\mathcal S$, we are considering elements up to \textit{squares} of units). Counting {all} ideals $I$ of norm $N(I)\leq\Delta$, it is quite easy to see that their number is  $\ll\Delta(\log\Delta)^{d-1}$: 	 
		
		Let $a_n$ be the number of ideals in $\co_K$ of norm $n$ and let $b_n$ be defined by $\zeta(s)^d=\sum_{n\geq 1} b_nn^{-s}$ (where $\zeta(s)$ is the Riemann zeta-function). 
		For each rational prime $p$, there are at most $d$ prime ideals that divide $p$ (and whose norm is a power of $p$), and so by comparing the Euler products of $\zeta_K(s)=\sum_{n\geq 1} a_nn^{-s}$ and $\zeta(s)^d$, we see that $a_n\leq b_n$ for all $n$.
		Thus it suffices to obtain an upper bound for $\sum_{1\leq n\leq\Delta} b_n$, which by the Tauberian theorem is easily seen to be $\ll \Delta(\log\Delta)^{d-1}$ as $\Delta\rightarrow\infty$ (where the implied constant depends only on $d$), as we wanted. (For the required background in analytic number theory, see, e.g., \cite[Chapter~7 and Appendix~II]{Na}).

As for the Pythagoras number $P$, in \cite[Corollary~3.3]{KY1} we proved that $P$ is bounded from above 
by a bound that depends only on the degree $d=[K:\Q]$. Specifically, the proof of \cite[Corollary~3.3]{KY1} established that $P\leq g(d)$, where $g(d)$ is the \textit{$g$-invariant}, i.e.,
the smallest rational integer
such that any quadratic form with $\Z$-coefficients of rank $d$ that is represented by the sum of any number of squares is represented by the sum of $g(d)$ squares.
Recently Beli, Chan, Icaza, and Liu \cite[Theorem~1.1]{BC+} showed an upper bound for the $g$-invariant that is exponential in $\sqrt d$, i.e., $g(d)\le c_1 \exp(c_2\sqrt{d})$. In any case, as our implied constant depends on $d$, we have $P\ll 1$.
	\end{proof}
	
	Note that the number of ideals in $\co_K$ of norm $\leq X$ grows as $\asymp X$ as $X\rightarrow\infty$ (and an analogue of this is known even for \textit{principal} ideals). But here the constants do depend on $K$, and in fact, even using finer versions of this asymptotics, it seems that for $X=\Delta$, the error term is often larger than the main term. Therefore in Theorem \ref{thm:univ} we had to use the weaker bound $\ll \Delta(\log\Delta)^{d-1}$.
	
	\

	\section{Weak Lifting Problem}\label{sec:2}
	
	We are now ready to prove our first main theorem concerning the general lifting problem.

	\begin{theorem}\label{th:lifting}
		Let $F$ be a totally real number field, $L$ an $\co_F$-lattice, and $d,m\in\N$. 
		There are at most finitely many totally real number fields $K\supset F$ of degree $d=[K:\Q]$ such that $L\otimes\ok$
		represents all elements of $m\ok^+$. 
	\end{theorem}	
	
	\begin{proof}
		Let $C$ be the constant (for $F$ and $L$) from Lemma \ref{lem:RR} and take a totally real number field $K\supset F$ of degree $d=[K:\Q]$. Let $X>1$ be  such that
		\begin{itemize}
			\item $X>\house{\alpha_i}_F=\house{\alpha_i}_K$ for all elements $\alpha_i$ in a fixed integral basis of $F$, and
			\item $X>3Cm$. 
		\end{itemize}
		
		Consider the subfield $E\subset K$ generated (as a field) by all the elements $\alpha\in \ok$ such that $\house{\alpha}_K<X$. By our choice of $X$, we have $F\subset E$.
		There are only finitely many algebraic integers of bounded degree and house (in particular, of degree $\leq d$ and house $<X$), and so they generate only finitely many possible number fields $E$. In particular,  
		there are (at most) finitely many possible fields $K$ for which $E=K$ can happen. Excluding these $K$s, we can assume that $E$ is a proper subfield of $K$. 
		
		Let $Y\geq X$ be the  
		minimum of $\house{\alpha}=\house\alpha_K$ for $\alpha\in \ok\setminus\co_E$ (the minimum exists, as the set of algebraic integers of bounded degree and $\house{\alpha}$ is finite). Let $\alpha\in \ok\setminus\co_E$ be such that $\house{\alpha}=Y$, and let $k\in\Z$ be the smallest rational integer such that $k+\alpha\succ 0$. We have $\house{k+\alpha}< 2\house\alpha+1\leq 3\house{\alpha}$.
		
		If $L\otimes\ok$ represents all elements from $m\ok^+$, 
		then there exists $v\in L\otimes\ok\subset K^r$ such that $Q(v)=m(k+\alpha)$. 
		If $v=(v_j)\in E^r\subset K^r$, then $Q(v)\in E$ (as $(L,Q)$ is an $\co_F$-lattice and $F\subset E$), and so also $\alpha\in E$, a contradiction. Thus there is $j$ such that $v_j\not\in E$, and so $\house{v}\geq \house{v_j}\geq Y$.
		
		By Lemma \ref{lem:RR} we then have \[Y^2\leq \house{v}^2\leq C\house{Q(v)}=Cm\house{k+\alpha}<3Cm\house{\alpha}=3YCm.\] 
		Thus $X\leq Y<3Cm$, which is a contradiction with our choice of $X$.	
	\end{proof}
	
	Note that one could also consider the natural extension of Theorem \ref{th:lifting} to representations of $\beta\co_K^+$ for fixed $\beta\in\co_F^+$. The corresponding statement follows immediately from the theorem, as $\beta\co_K^+\subset m\co_K^+$ for $m=N_{F/\Q}(\beta)$.
	
	\begin{cor}\label{cor:sums of squares}
		For each $m,d\in\N$, there are only finitely many totally real number fields $K$ of degree $d=[K:\Q]$ such that all elements in $m\ok^+$ are sums of squares.	
	\end{cor}
	
	\begin{proof}
		The Pythagoras number of the ring of integers is bounded in terms of degree of the field extension \cite[Corollary~3.3]{KY1}, i.e., \textit{there is a function $g(d)$ such that $\mathcal P(\co_F)\leq g(d)$ whenever $[F:\Q]=d$}. Thus, it suffices to consider the sum of a bounded number of squares $Q=x_1^2+\dots+x_{g(d)}^2$. Then the corollary is immediate from the previous theorem
		applied to the free $\co_F$-lattice $(\co_F^{g(d)},Q)$. 
		(See the end of the proof of Theorem \ref{thm:univ} for more information on $g(d)$.)
	\end{proof}
	 
 We get the subsequent strong addendum to Theorems 1.1 and 1.2 in \cite{KY1}:
	\begin{cor}\label{cor:z-forms}
		For each $d, r\in\N$, there are only finitely many totally real number fields $K$ of degree $d=[K:\Q]$ such that there is a positive definite quadratic form over $\Z$ of rank $r$ that is universal over $K$.
	\end{cor}
	
	\begin{proof}
		Let $Q$ be a positive definite quadratic form over $\Z$  of rank $r$. By Theorem 1 in \cite{CS1} (possibly first taking $2Q$ to make the form classical), \textit{there exists $m\in\Z$ (depending on $r$) such that $mQ$ is the sum of squares of linear forms with $\Z$-coefficients.}
		Thus, if $Q$ is universal over $K$, then every element of $m\co_K^+$ is the sum of squares. By Corollary \ref{cor:sums of squares} there are only finitely many such fields of given degree $d$.	
	\end{proof}

	\section{Weak Kitaoka's conjecture}
	
	Let us use the results obtained in the previous section to prove our Theorem \ref{thm:2}. However, first we need to recall a (mostly well-known) fact concerning sublattices.

	\begin{lemma}\label{lemma:prep} Let $F$ be a totally real number field.
		Let  $d, r\in\N$ and let $\ell$ be an $\co_F$-lattice of rank $r$. There exists a positive integer $m$ such that for every totally real extension $K\supset F$ of degree $d$ and $\co_K$-lattice $M$ of rank $r$ satisfying $\ell\otimes \co_K\subset M$ we have $mM\subset \ell\otimes \co_K$.
	\end{lemma}
	\begin{proof}
		Let $L_K=\ell\otimes \co_K$.  
		By definition, we have $(\mathfrak{v}\ell)\co_K=\mathfrak{v}L_K$, and \cite[82:11]{O1} gives
		\begin{equation}\label{eq:82:11}
		(\mathfrak{v}\ell)\co_K=\mathfrak{v}L_K=\mathfrak{a}^2\mathfrak{v}M,
		\end{equation} 
		where $\mathfrak{a}$ is an integral ideal (equal to the product of the \textit{invariant factors}  
		of $L_K$ in $M$ \cite[\S 81D]{O1}).  	
		As we assume that all lattices satisfy $\mathfrak{n}M\subset\co_K$, we can use \cite[82:19]{O1} to obtain $\mathfrak v(2M)=2^r(\mathfrak{v}M)\subset \co_K$.
		
		By \eqref{eq:82:11}, the integral ideal $\mathfrak v(2M)$ divides $2^r (\mathfrak{v}\ell)\co_K$, and so there are only finitely many possibilities for it.
		\cite[103:4]{O1} says that \textit{there are only finitely many $\ok$-lattices of given volume and integral scale}. As the $\ok$-lattice $2M$ has integral scale and bounded volume, there are also only finitely many possibilities for $M$.
		
		Thus there are only finitely many possibilities for the (subgroup) index $(M:L_K)$; taking $m$ to be the least common divisor of all the possible indices, we get 
		$mM\subset (M:L_K)M\subset L_K$
		as we wanted.
	\end{proof}

	Our final tool will be the notion of a ``universality criterion set'':
	
	\begin{definition}		
		Let $K$ be a totally real number field. A \textit{universality criterion set in $K$} is a finite set $S_K\subset \ok^+$ such that if an $\ok$-lattice represents all elements of $S_K$, then it is universal.
	\end{definition} 
	
	Note that in the definition, $S_K$ is not required to be minimal or unique.
	As an example, let us mention that by the 290-Theorem by Bhargava and Hanke \cite{BH} we can take  $S_\Q=\{1,\dots,290\}$.
	
	The existence of a universality criterion set in any totally real number field was probably first mentioned by Kim, Kim, and Oh \cite{KKO} (see also \cite[Section 6]{Km}), 
	with the name coming from \cite{EKK}. The existence was recently proved by Chan and Oh \cite[Corollary 5.8]{CO}.
	
	For a matrix $M$, let $K(M)$ be the number field extension of $K$ generated by all the entries of $M$.
	
	\begin{theorem}\label{th:kitaoka}
		For each $d\in\N$, there are only finitely many totally real number fields $K$ of degree $d=[K:\Q]$ over which there is a ternary universal $\ok$-lattice.
	\end{theorem}
	
	\begin{proof} Most of the proof will consist of 
		inductively defining finite sets $\mathcal A_i, \mathcal F_i, \mathcal H_i$ (throughout the proof, we view also the empty set as finite) 
		of number fields of degree $\leq d$ 
		with the following property: 
		
		\textit{If $K$ is a number field of degree $\leq d$ 
			that admits a ternary universal $\ok$-lattice $L$, then 
			\begin{enumerate}
				\item $K\in\mathcal A_i$, or
				\item there is $F\in\mathcal H_i, F\subsetneq K$ with a ternary universal $\co_F$-lattice $\ell$ such that $\ell\otimes\co_K\subset L$, or
				\item there are $k_j\in \mathcal F_j, j=0, \dots, i,$ such that $k_0\subsetneq k_1\subsetneq\dots\subsetneq k_i\subsetneq K$.
		\end{enumerate}	}	
		
		Let us denote the set of number fields $K$ satisfying item \textit{(2)} (item \textit{(3)}, resp.) by $\mathcal B_i$ ($\mathcal C_i$, resp.). We do not claim (yet) that $\mathcal B_i$ or $\mathcal C_i$ is necessarily finite.
		
		\medskip
		
		Let us start the construction with		
		$\mathcal A_0=\mathcal H_0=\emptyset$ and $\mathcal F_0=\{\Q\}$. Note that if $K$ is a number field with a ternary universal $\ok$-lattice, then $K\neq\Q$, and so $K$ satisfies item \textit{(3)} for $i=0$, as $\Q=k_0\subsetneq K$.

		Let further $K\not\in\mathcal A_i$ admit a ternary universal $\ok$-lattice $L$.
		We do not need to consider fields $K\in\mathcal B_i$ (i.e., satisfying item \textit{(2)} above), for we will have $\mathcal H_{i+1}\supset\mathcal H_i$ by our construction, and so $\mathcal B_i\subset\mathcal B_{i+1}$.
		
		Thus consider $K\in\mathcal C_i$ and accordingly take some $k=k_i\in\mathcal F_i$ such that $k\subsetneq K$.
		Let $S_k$ be the corresponding universality criterion set. The universal $\co_K$-lattice $L$ represents all elements $s\in S_k$, i.e., there are $v_s\in L$ such that $Q(v_s)=s$. Let $M=(B(v_s,v_t))$ be the corresponding Gram matrix. 
		Further let $B(v_s,v_t)=m_{st}/2$ for $s\neq t$ and $m_{ss}=s\in S_k$ so that $m_{st}\in\ok$ for all $s,t$.
		
		As $L$ is totally positive definite, 
		for each $s,t\in S_k$ we have the Cauchy--Schwarz inequality $m_{st}^2\preceq 4st$, 
		and so $\Tr_{K/\Q}(m_{st}^2)\leq \max_{s,t\in S_k} 4\Tr_{K/\Q}(st)\leq \max_{s,t\in S_k} 4d\Tr_{k/\Q}(st)$ is bounded. Since $m_{st}$ is an algebraic integer of degree $\leq d$, there are only finitely many possibilities for it (that do not depend on $K$). Thus there are also only finitely many possibilities for the field $k(M)\subset K$.
		
		\medskip
		
		Let us next distinguish three cases according to whether we have equality in any of the inclusions $k\subset k(M)\subset K$:
		
		a) To deal with the case $k(M)=K$, 
		\textit{let $\mathcal A_{i+1}$ be the union of $\mathcal A_i$ with all these fields $k(M)$ as $k$ runs through $\mathcal F_i$.}
		Therefore if moreover $K\not\in\mathcal A_{i+1}$, then $k(M)$ is a proper subfield of $K$.
		
		b) 
		Assume now that $k(M)=k$ and consider the $\co_k$-span $\ell$ of all vectors $v_s$ for $s\in S_k$ (as a subset of $L$). Then $\ell$ (equipped with the restrictions of $B,Q$) is an $\co_k$-lattice (for all entries of the Gram matrix $M$ lie in $k$) that represents  
		all elements of $S_k$. By the assumption that $S_k$ is a universality criterion set for $k$, the $\co_{k}$-lattice $\ell$ is {universal} over $k$.
		
		Then $\ell\otimes \ok$ is a sublattice of the ternary lattice $L$, and so the rank of $\ell$ is also 3 (no universal lattice has rank $\leq 2$, and rank of $\ell\leq$ rank of $L=3$).
		Thus we can 
		\textit{let $\mathcal H_{i+1}$ be the union of $\mathcal H_i$ with all these fields $k(M)=k$ as $k$ runs through $\mathcal F_i$.}
		
		c)
		Finally, we are left with the case $k\subsetneq k(M)\subsetneq K$. Then \textit{let $\mathcal F_{i+1}$ be the set of these fields $k_{i+1}=k(M)$}, i.e.,  $k\subsetneq k_{i+1}=k(M)\subsetneq K$.
		
		\medskip
		
		Thus we have defined all the needed sets for $i+1$.		
		As there were only finitely many possibilities for $k(M)$ (which did not depend on $K$ but only on $k\in\mathcal F_i$), each of the new sets $\mathcal A_{i+1}$, $\mathcal F_{i+1}$, $\mathcal H_{i+1}$ is still finite, as we wanted.
		
		\medskip
		
		Having constructed the desired sets $\mathcal A_i, \mathcal F_i, \mathcal H_i$, let us now consider them when $i=d$.
		Note that no field $K$ satisfies \textit{(3)} for $i=d$ (i.e., $\mathcal C_d=\emptyset$), for otherwise we would have 
		$d=[K:\Q]\geq [k_d:\Q]+1\geq [k_{d-1}:\Q]+2\geq\dots\geq [k_0:\Q]+d+1$, which is impossible.
		
		There are finitely many fields $K$ in the finite set $\mathcal A_d$, and so it remains to consider fields $K\in\mathcal B_d$. 
		We will prove that there are also finitely many of them.

		The set $\mathcal H_d$ is finite, and by Corollary 1 in \cite{Ea}, there are only finitely many ternary universal lattices over a fixed field $F\in\mathcal H_d$. Thus there are finitely many pairs $(F, \ell)$ that can appear in item~\textit{(2)}. 
		
		By Lemma \ref{lemma:prep} there is $m\in\N$ depending on $(F,\ell)$ and $d$, but not on the specific field $K$ or lattice $L$, such that if $\ell\otimes \ok\subset L$ (which is true by \textit{(2)}), then $mL\subset \ell\otimes \ok$. 
		
		As $L$ is universal, $\ell\otimes \ok$ represents all elements of $m\ok^+$. But by Theorem \ref{th:lifting}, there are only finitely many such fields $K$ (and for each of them, there are again at most finitely many ternary universal lattices $L$). Thus, for a fixed pair $(F,\ell)$, there are at most finitely many extensions $(K,L)$ (of given degree $d$). Thus also item \textit{(2)} gives only finitely many fields $K$, concluding the proof.
	\end{proof}
	
	Finally, note that essentially the same proof yields the following corollary for fields of odd degree (over which there never exists a ternary universal lattice). For \textit{classical} lattices, this result was proved by Kim \cite{Ki3}.
	
	\begin{cor}\label{cor:kitaoka}
		For each \emph{odd} $d\in\N$, there are only finitely many totally real number fields $K$ of degree $d=[K:\Q]$ over which there is a quaternary universal $\ok$-lattice.
	\end{cor}
	
	\begin{proof}
		The proof is almost verbatim the same as the proof of Theorem \ref{th:kitaoka}, replacing ``ternary'' by ``quaternary'' everywhere in the proof (and adding the requirement that all number fields considered have odd degree).
		As $K$ is assumed to have odd degree, all its subfields also have odd degree, and so do not admit a ternary universal lattice \cite[Lemma~3]{EK}. Thus the lattice $\ell$ considered in part b) of the preceding proof must have rank 4.	
		Finally, towards the end of the proof one uses \cite[Theorem~1]{EK} instead of \cite[Corollary~1]{Ea}.
	\end{proof}

\end{document}